\documentclass[11pt]{article}
\usepackage{amsmath,amssymb,amsthm}
\theoremstyle{definition}
\newtheorem{Def}{Definition}[section]
\newtheorem{Thm}[Def]{Theorem}
\newtheorem{Cor}[Def]{Corollary}
\newtheorem{Prop}[Def]{Proposition}
\newtheorem{Lem}[Def]{Lemma}
\newtheorem{Cl}{Claim}
\newtheorem*{Cl*}{Claim}
\newtheorem{Ex}{Example}
\newtheorem*{Rem}{Remark}
\newcommand{\Hom}{{{\rm{Hom}}}}
\newcommand{\Int}{{{\rm{Int}}}}
\newcommand{\conv}{{{\rm{conv}}}}
\newcommand{\HH}{{{\rm{H}}}^+}
\newcommand{\OO}{{{\rm{O}}}}
\newcommand{\PP}{{{\rm{P}}}}
\newcommand{\QQ}{{{\rm{Q}}}}
\newcommand{\cc}{{{\rm{c}}}}
\newcommand{\QQQ}{{\mathbb{Q}}}
\newcommand{\RRR}{{\mathbb{R}}}
\newcommand{\ZZZ}{{\mathbb{Z}}}
\newcommand{\aaa}{{\mathfrak{a}}}

\newcommand{\mmm}{{\mathfrak{m}}}
\title{Formulas of F-thresholds and F-jumping coefficients on toric rings
}
\date{}
\author{Daisuke Hirose}
\begin{document}
\maketitle
\begin{abstract}
Musta\c{t}\v{a}, Takagi and Watanabe
define F-thresholds,
which are invariants of a pair of ideals in a ring of characteristic $p > 0$.
In their paper,
it is proved that F-thresholds are equal to 
jumping numbers of test ideals on
 regular local rings.
In this note, we give formulas of F-thresholds and F-jumping
 coefficients on toric rings.
By these formulas,
we prove that there exists an inequality between F-jumping coefficients
and F-thresholds.
In particular, we observe a comparison between F-pure thresholds 
and F-thresholds in some cases.
As applications,
we give a characterization of regularity for toric rings defined
 by simplicial cones,
and we prove the rationality of F-thresholds in some cases.
\end{abstract}
\section{Introduction}
Let $R$ be a commutative Noetherian ring of characteristic $p>0$.
In \cite{HY}, Hara and Yoshida defined a generalized test ideal
$\tau(\aaa^c)$ of an ideal $\aaa \subseteq R$ and 
a positive real number $c \in \RRR_{>0}$.
This is a generalization of the test ideal $\tau(R)$,
which appeared in the theory of tight closure 
(cf.\ \cite{HH}).
On the other hand,
this ideal is a characteristic $p$ analogue of a multiplier ideal
(cf.\ \cite{Laz}).
Similarly, one can define a characteristic $p$ analogue of
a jumping coefficient of a multiplier ideal,
which is called the F-jumping coefficient.
In other words, $c \in \RRR_{>0}$ is an F-jumping coefficient of 
an ideal $\aaa \subseteq R$ 
if $\tau(\aaa^c) \neq \tau(\aaa^{c-\varepsilon})$ 
for all $\varepsilon > 0$.

Musta\c{t}\v{a}, Takagi and Watanabe studied 
an F-jumping coefficient.
In \cite{MTW}, they defined another invariant of singularities,
which is called the F-threshold.
They proved that
an F-threshold coincides with an F-jumping coefficient
on a regular local ring of characteristic $p>0$.
Using this relation,
they proved basic properties of F-jumping coefficients.
Blickle, Musta\c{t}\v{a} and Smith studied F-jumping coefficients
or F-thresholds on F-finite regular rings.
In particular, they proved the rationality and discreteness
of \mbox{F-thresholds} for F-finite regular rings under some assumptions
(cf.\ \cite{BMS1} and \cite{BMS2} for details),
which partially solves an open problem in \cite{MTW}.

However, if rings have singularities,
F-thresholds do not coincide with F-jumping coefficients in general.
In \cite{HMTW},
Huneke, Musta\c{t}\v{a}, Takagi and Watanabe studied 
various topics of F-thresholds for general settings.
For example,
they defined a new invariant called the F-threshold of a module,
which coincides with an F-jumping coefficient for F-finite and F-regular
local normal $\QQQ$-Gorenstein rings.
As a corollary,
they proved 
an inequality between the F-threshold and the \mbox{F-pure} threshold,
which is the smallest \mbox{F-jumping} coefficient for a fixed ideal.
They also gave examples of non-regular rings and ideals
whose F-thresholds coincide with their F-pure thresholds.

In this paper, we consider F-thresholds and F-jumping coefficients
of monomial ideals for toric rings,
which are not necessarily regular.
We give the explicit formula of F-thresholds in section 3,
which is written in terms of
 cones corresponding to toric rings and Newton polyhedrons
corresponding to monomial ideals.
Using this formula,
we compare F-thresholds with \mbox{F-jumping} coefficients in section 4.
As applications,
we give a characterization of regularity of toric rings
defined by simplicial cones in Theorem \ref{charareg}.
We also prove the rationality of F-thresholds
of monomial ideals for toric rings defined by
simplicial cones in Theorem \ref{rational}.

\section{The definition of F-thresholds}
Throughout this paper, we assume that every ring $R$ is reduced, 
and contains a perfect field $k$ whose characteristic is $p > 0$.
Let $F:R \to R$ be the Frobenius map which sends an element $x$ 
of $R$ to $x^p$.
For a positive integer $e$,
the ring $R$ viewed as an $R$-module 
via the $e$-times iterated Frobenius map is denoted by $\mbox{}^e R$.
We assume that a ring $R$ is F-finite,
that is, $\mbox{}^1R$ is a finitely generated \mbox{$R$-module.}
We also assume that a ring $R$ is F-pure,
that is, the Frobenius map $F$ is pure.
For an ideal $J$ and a positive integer $e$,
$J^{[p^e]}$ is the ideal generated
by $p^e$-th power elements of $J$.
For example, if $J$ is $(X_1, X_2^2) \subset k[X_1, X_2]$,
then $J^{[p^e]}$ is $(X_1^{p^e}, X_2^{2p^e})$.
We recall the definition and some remarks of \mbox{F-thresholds}
which are defined by Musta\c{t}\v{a}, Takagi and Watanabe
in \cite{MTW}.
These are invariants of a pair of ideals.

\begin{Def}[{F-threshold, cf.\ \cite[\S 1]{MTW}}]
Let $\aaa$ and $J$ be nonzero proper ideals of a ring $R$
such that $\aaa \subseteq \sqrt{J}$.
The $p^e$-th threshold $\nu_{\aaa}^J(p^e)$
of $\aaa$ with respect to $J$ 
is defined as
$$\nu_{\aaa}^J(p^e):= \max \{ r \in \mathbb{N}|
 \aaa^r \nsubseteq J^{[p^e]}\}.$$
Then we define the F-threshold $\cc^J(\aaa)$
of $\aaa$ with respect to $J$ as
$${\cc}^J(\aaa):= \lim_{e \to \infty}
 \frac{\nu_{\aaa}^J(p^e)}{p^e}.$$
\end{Def}

\begin{Rem}
Since $R$ is F-pure,
if $u \notin J^{[p^e]}$, then $u^{p} \notin J^{[p^{e+1}]}$.
This implies that $\nu_{\aaa}^J(p^e)/p^e \le
 \nu_{\aaa}^J(p^{e+1})/p^{e+1}$,
and hence $\cc^J(\aaa)$ exists under our assumption.
Furthermore, the assumption that $\aaa \subseteq \sqrt{J}$
implies that $\cc^J(\aaa) < \infty$.
However, in general, 
this limit does not necessarily exist.
In \cite{HMTW}, Huneke, Musta\c{t}\v{a}, Takagi and Watanabe defined
$\cc_{-}^J(\aaa)$ and $\cc_{+}^J(\aaa)$ as
$$\cc_{-}^J(\aaa):=\lim \inf  \frac{\nu_{\aaa}^J(p^e)}{p^e},\
\cc_{+}^J(\aaa):= \lim \sup \frac{\nu_{\aaa}^J(p^e)}{p^e},$$
for ideals $\aaa$ and $J$ with $\aaa \subseteq \sqrt{J}$.
When $\cc_{-}^J(\aaa)=\cc_{+}^J(\aaa)$,
they call it the F-threshold of $\aaa$ with respect to $J$,
which is denoted by $\cc^J(\aaa)$.
They give a sufficient condition when $\cc^J(\aaa)$ exists
(cf.\ \cite[Lemma 2.3]{HMTW}).
\end{Rem}

Let $R^{\circ}$ be the set of elements of $R$ which are not contained 
in any minimal prime ideals of $R$.
Let $\aaa$ be an ideal of $R$ 
such that $\aaa \cap R^{\circ} \neq \emptyset$,
and let $c$ be a positive real number.
For an $R$-module $D$, 
we define the $\aaa^c$-tight closure of 
the zero submodule in $D$ as the following,
which is denoted by $0_D^{*\aaa^c}$.
For $z \in D$,
an element $z$ is contained in  $0_D^{*\aaa^c}$
if there exists $x \in R^{\circ}$ such that
$$x \aaa^{\lceil c p^e \rceil} (1 \otimes z) = 0 
\in \mbox{}^e R \otimes D,$$
where $e$ runs all sufficiently large positive integers.

\begin{Def}[test ideal]
Let $\aaa \subseteq R$ be an ideal such that 
$\aaa \cap R^{\circ} \neq \emptyset$,
and $c$ a positive real number.
Let $E:= \oplus_{\mmm} E_R(R /\mmm)$,
where $\mmm$ runs all maximal ideals of $R$
and $E_R(R/ \mmm)$ is the injective hull 
of the residue field $R/\mmm$.
The test ideal $\tau(\aaa^c)$ of $\aaa$ and $c$
is defined as
$$\tau(\aaa^c) := \bigcap_{D \subseteq E}
{\rm{Ann}}_R 0_D^{*^{\aaa^c}},$$
where $D$ runs all finitely generated $R$-submodules of $E$.
\end{Def}

In \cite{MTW}, they also proved the connection between F-thresholds and
test ideals for regular local rings.
Moreover, in \cite{BMS2}, they generalized it for
regular rings.

\begin{Thm}[{\cite[Proposition 2.7]{MTW}} and {\cite[Proposition
 2.23]{BMS2}}]\label{MTWprop27}
Let $\aaa$ and $J$ be proper ideals on a regular ring $R$
such that $\aaa \subseteq \sqrt{J}$.
Then
$$\tau(\aaa^{\cc^J(\aaa)}) \subseteq J.$$
On the other hand, for a positive real number $c$, we have
$\aaa \subseteq \sqrt{\tau(\aaa^c)}$, and also
$$\cc^{\tau(\aaa^c)}(\aaa) \le c.$$
In addition, there exists a map from the set of
 F-thresholds of $\aaa$ to the set of test ideals of $\aaa$
which sends the test ideal $J$ to $\cc^J(\aaa)$.
Moreover, this map is bijective.
The inverse map sends an F-threshold $c$ of $\aaa$ to
 $\tau(\aaa^c)$.
\end{Thm}

By the two inequalities in Theorem \ref{MTWprop27},
F-thresholds on a regular ring are equal to 
F-jumping coefficients.
They are analogues of jumping coefficients of a multiplier ideal.


\begin{Cor}\label{threisjump}
For a fixed nonzero proper ideal $\aaa$ 
on a regular ring $R$,
the set of F-thresholds of $\aaa$ is equal to the set of F-jumping
coefficients of $\aaa$.
\end{Cor}


\section{A formula of F-thresholds on toric rings}
Let us begin with fixing the notation about toric geometries.
Let \mbox{$N \cong \ZZZ^d$}
and \mbox{$M \cong \Hom_{\ZZZ}(N, \ZZZ)$} which is isomorphic to
$\ZZZ^d$.
The duality pair of \mbox{$M_{\RRR}:=M \otimes_{\ZZZ}\RRR$}
and $N_{\RRR}:=N \otimes_{\ZZZ}\RRR$ is denoted by
$$\langle \ , \ \rangle :M_{\RRR} \times N_{\RRR}
\to \RRR.$$
For a strongly convex rational polyhedral cone $\sigma$
of $N_{\RRR}$,
we define
$$\sigma^{\vee}:=\{u \in M_{\RRR}|
\langle u, v \rangle \ge 0, \forall v \in \sigma\}.$$
Let $R$ be a toric ring defined by $\sigma$, that is,
the subalgebra of Laurent polynomial $k[X_1^{\pm 1}, \cdots, X_d^{\pm 1}]$
generated by sets $\{X^u | u \in \sigma^{\vee} \cap M\}$,
 where $X^u$ expresses $X_1^{u_1}\cdots X_d^{u_d}$
for $u=(u_1, \cdots , u_d) \in M$.
Since we always assume that $k$ is a perfect field,
a toric ring is F-finite under our assumption.
A proper ideal $\aaa$ of $R$ is said to be 
a monomial ideal if $\aaa$
is generated by monomials of $R \subset k[X_1^{\pm 1},\cdots , X_d^{\pm 1}]$. 
For a monomial ideal $\aaa$,
we define two types of sets in $\sigma^{\vee}$.

\begin{Def}
The Newton polyhedron $\PP(\aaa)$ of $\aaa$
is defined as
$$\PP(\aaa):=\conv 
\{u \in M|X^u \in \aaa\}.$$
Moreover, we define
$$\QQ(\aaa):=\bigcup_{X^u \in \aaa}
 u + \sigma^{\vee}.$$
In addition, for $\PP(\aaa)$ and a positive
real number $\lambda$, the sets $\lambda \PP(\aaa)$
is defined as
$$\lambda \PP(\aaa) := \{\lambda u \in M_{\RRR}
|u \in \PP(\aaa)\}.$$
We can define $\lambda \QQ(\aaa)$ by the same
way. 
\end{Def}

The following proposition is basic properties of $\QQ(\aaa)$ and 
$\PP(\aaa)$,
which follows immediately.

\begin{Prop}\label{pro_of_PQ}
Let $\aaa$ be a monomial ideal of a toric ring $R$ defined by 
a cone $\sigma \subseteq N_{\RRR}$.
\begin{enumerate}
\item[(i).] For $e \in \ZZZ_{>0}$, it holds that
$\QQ(\aaa)=(1/p^e)\QQ(\aaa^{[p^e]})$.
\item[(ii).] $\PP(\aaa)+ \sigma^{\vee}
\subseteq \PP(\aaa)$.
\item[(iii).] If $\aaa=(X^{\mathbf{a}_1},\cdots,
X^{\mathbf{a}_s})$,
then \mbox{$\PP(\aaa)=
\conv\{\mathbf{a}_1,\cdots,\mathbf{a}_s\} + \sigma^{\vee}$}.
\end{enumerate}
\end{Prop}


Using this notation,
we give a computation of F-thresholds in real affine geometries.
This formula is a generalization of \cite[Eample 2.7]{HMTW}.
Let $R$ be a toric ring defined by a cone 
$\sigma \subseteq N_{\RRR}$.
Let $\aaa$ be a monomial ideal.
For $u \in \sigma^{\vee}$, we define $\lambda_{\aaa}(u)$ as
\begin{eqnarray*}
\lambda_{\aaa}(u):= \left\{\begin{array}{ll}
\sup\{\lambda \in \RRR_{\ge 0} |
u \in \lambda \PP(\aaa)\} &
(\exists \lambda \in \RRR_{>0} \ s. t.\
 u \in \lambda \PP(\aaa)),\\
0 & (\forall \lambda \in \RRR_{>0},\
 u \notin \lambda \PP(\aaa)).
\end{array} \right.
\end{eqnarray*}

\begin{Thm}\label{Thmfthre}
Let $R$ and $\aaa$ be as the above.
Let $J$ be a monomial ideal such that $\aaa \subseteq \sqrt{J}$.
Then
$$\cc^J({\aaa})= \sup_{u \in \sigma^{\vee} \setminus
 \QQ(J)} \lambda_{\aaa}(u).$$
\end{Thm}

\begin{proof}
We assume that
$\aaa=(X^{\mathbf{a}_1}, \cdots, X^{\mathbf{a}_s})$ 
where $\mathbf{a}_i \in M$ for $i=1, \cdots, s$.
To prove the theorem, we need the following two claims.
\begin{Cl}
For any positive integers $e$, there exists 
$u \in \sigma^{\vee} \setminus \QQ(J)$ such that
$\nu_{\aaa}^J(p^e)/p^e \leq \lambda_{\aaa}(u)$. 
\end{Cl}
\begin{Cl}
For every element $u \in \sigma^{\vee} \setminus \QQ(J)$, there
 exists a positive integer $e$ such that 
$\nu_{\aaa}^J(p^e) / p^e \ge \lambda_{\aaa}(u)$.
\end{Cl}
We note that Claim 1 implies
$\cc^J(\aaa) \le \sup \lambda_{\aaa}(u)$.
Since the definition of right-hand side supremum,
$\nu_{\aaa}^J(p^e) / p^e \leq \sup \lambda_{\aaa}(u)$.
Thus
$\cc^J(\aaa) \leq \sup \lambda_{\aaa}(u)$
by the definition of F-thresholds 
and the fact that a supremum is the minimum number in upper bounds.
By the similar argument, Claim 2 implies
\mbox{$\cc^J(\aaa) \ge \sup \lambda_{\aaa}(u)$.}
\begin{proof}[Proof of claim 1.]
We fix a positive integer $e$.
Since the definition of the \mbox{$p^e$-th} threshold,
for every $i=1, \cdots, s$, there are nonnegative integers $r_i$
such that $\sum r_i = \nu_{\aaa}^J(p^e)$
and $X^{\sum r_i \mathbf{a}_i} \notin J^{[p^e]}$.
In particular, $\sum r_i \mathbf{a}_i \notin \QQ(J^{[p^e]})$.
This is equivalent to 
$(1/p^e) \sum r_i \mathbf{a}_i \notin (1/p^e) \QQ(J^{[p^e]})$.
By Proposition \ref{pro_of_PQ} (i),
we have $(1/p^e) \sum r_i \mathbf{a}_i \notin  \QQ(J)$.
Hence
\begin{equation*}
\frac{1}{p^e} \sum r_i \mathbf{a}_i
=\frac{\nu_{\aaa}^J(p^e)}{p^e}
\sum \frac{r_i}{\nu_{\aaa}^J(p^e)} \mathbf{a}_i
\end{equation*}
which is an element of $(\nu_{\aaa}^J(p^e)/p^e) 
\PP(\aaa)$.
Thus $\nu_{\aaa}^J(p^e)/p^e 
\le \lambda_{\aaa}((1/p^e) \sum r_i \mathbf{a}_i)$.
\end{proof}
\begin{proof}[Proof of Claim 2.]
We fix $u \in \sigma^{\vee} \setminus \QQ(J)$,
which satisfies $\lambda_{\aaa}(u) \neq 0$.
We find an integer $e$ which satisfies the assertion of Claim 2 by
 three steps.
\newline
{\sc Step 1.}
We prove that
there exists an element $u'$ on the boundary
$(\lceil p^e \lambda_{\aaa}(u) \rceil/p^e) \PP(\aaa)$
such that $u' \notin \QQ(J)$ for sufficiently large $e$.
The following sequence of real numbers
$$ \lambda_{\aaa}(u) \le \cdots 
\le \frac{\lceil p^{e+1} \lambda_{\aaa}(u) \rceil}{p^{e+1}} 
\le \frac{\lceil p^e \lambda_{\aaa}(u) \rceil}{p^e} 
\le \cdots 
\le \frac{\lceil p\lambda_{\aaa}(u) \rceil}{p}$$
induces the sequence of Newton polyhedrons
$$\frac{\lceil p \lambda_{\aaa}(u) \rceil}{p}\PP(\aaa)
\subseteq \cdots \subseteq
\frac{\lceil p^e \lambda_{\aaa}(u) \rceil}{p^e}\PP(\aaa)
\subseteq \frac{\lceil p^{e+1} \lambda_{\aaa}(u) \rceil}{p^{e+1}}\PP(\aaa)
\subseteq \cdots \subseteq 
\lambda_{\aaa}(u)\PP(\aaa).$$
In particular, 
the above sequences are strict
if $\lambda_{\aaa}(u) \notin (1/p^e) \ZZZ$ for all $e$.
Since $u \notin \QQ(J)$,
we can find such $u'$ by taking $e$ sufficiently large.
\newline
{\sc Step 2.}
We prove that there exist nonnegative integers $r_i$ 
for every \mbox{$i=1, \cdots, s$}
such that $\sum r_i /p^e \ge \lambda_{\aaa}(u)$
and $u'':= \sum r_i \mathbf{a}_i /p^e \notin \QQ(J)$.
Since $u'$ is contained in
$(\lceil p^e \lambda_{\aaa}(u) \rceil/p^e)\PP(\aaa)$,
$u'$ can be written 
$$\frac{\lceil p^e \lambda_{\aaa}(u) \rceil}{p^e} (\sum c_i \mathbf{a}_i
+ \omega),$$
where $c_i$ are nonnegative real numbers with $\sum c_i =1$
and $\omega \in \sigma^{\vee}$ by Proposition \ref{pro_of_PQ} (iii).
Let 
$$r_i:=\lceil \lceil p^e \lambda_{\aaa}(u) \rceil c_i \rceil.$$
Then
$$\sum \frac{r_i}{ p^e} \ge \frac{\lceil p^e\lambda_{\aaa}(u) \rceil}{p^e}
\sum c_i \ge \lambda_{\aaa}(u).$$
Moreover,
$$|u''+\frac{\lceil p^e \lambda_{\aaa}(u) \rceil}{p^e}\omega - u'| 
\le \sum |\frac{\lceil \lceil p^e \lambda_{\aaa}(u) \rceil c_i \rceil}{p^e} 
- \frac{\lceil p^e \lambda_{\aaa}(u)\rceil c_i}{p^e}| \cdot |\mathbf{a}_i|
< \frac{1}{p^e}\sum |\mathbf{a}_i|.$$
Since $u' \notin \QQ(J)$,
we have $u'' + (\lceil p^e \lambda_{\aaa}(u) \rceil /p^e)\omega 
\notin \QQ(J)$
if we choose $e$ sufficiently large.
By the definition of $\QQ(J)$,
we have $u'' \notin \QQ(J)$.
\newline
{\sc Step 3.}
Since $u'' \notin \QQ(J)$,
$$p^e u'' \notin p^e\QQ(J)=\QQ(J^{[p^e]}).$$
Therefore $X^{p^eu''} \notin J^{[p^e]}$.
On the other hand,
$X^{p^eu''} \in \aaa^{\sum r_i}$
by the construction of $u''$.
Therefore $\sum r_i \le \nu_{\aaa}^J(p^e)$.
This implies $\lambda_{\aaa}(u) \le \nu_{\aaa}^J(p^e)/p^e$.
\end{proof}
We complete the proof of Theorem \ref{Thmfthre}.
\end{proof}

\section{A comparison between F-jumping coefficients and F-thresholds}
F-pure thresholds are defined via F-singularities of the
 pair $(R, \aaa^c)$ where $c$ is a positive real number.
See \cite[Definition 1.3, Definition 2.1]{TW} for details.
Since F-finite toric rings are strongly F-regular,
the F-pure thresholds can be defined as follows
(See \cite[Proposition 2.2]{TW}).

\begin{Def}[F-pure thresholds]
Let $R$ be a toric ring, and $\aaa$ a monomial ideal.
The F-pure threshold $\cc(\aaa)$ of $\aaa$ is
 defined as
$$\cc(\aaa) := \sup \{ c \in \RRR_{\ge 0}|
 \tau(\aaa^c)=R\}.$$
\end{Def}

Hence the F-pure threshold of $\aaa$ 
is the smallest F-jumping coefficient of $\aaa$.
In \cite{HMTW},
the inequality between an F-pure threshold and an F-threshold 
on a local ring was given 
in terms of the F-threshold of a module (\cite[Section 4.]{HMTW}).
In this section, we consider the inequality on toric rings,
by a combinatorial method.
Furthermore, we consider the connection between arbitrary 
F-jumping coefficients and F-thresholds with respect to some monomial
ideals.
To compute F-pure thresholds and F-jumping coefficients of monomial ideals,
we introduce the following theorem
given by Blickle. 

\begin{Thm}[{\cite[Theorem 3]{B}}]\label{Bthm3}
Let $R$ be the toric ring defined by 
\newline
\mbox{$\sigma = \RRR_{\ge 0}v_1 +
 \cdots + \RRR_{\ge 0}v_n \subseteq N_{\RRR}:= \RRR^d$},
where $v_j \in N$ are primitive.
Then the test ideal $\tau(\aaa^c)$ of a monomial ideal $\aaa$ 
is also a monomial ideal.
Moreover, $X^u \in \tau(\aaa^c)$ for $u \in M$
if and only if there exists $\omega \in M_{\RRR}$ such that
\begin{alignat*}{1}
\langle \omega, v_j \rangle \le 1,\ j=1,\cdots,n, \\
u + \omega \in \Int(c \PP(\aaa)). 
\end{alignat*}
\end{Thm}

By this theorem,
the F-pure threshold of a monomial ideal 
on a toric ring can be described as the following corollary.

\begin{Cor}\label{fpuretoric}
Let $R$ and $\aaa$ be as in Theorem \ref{Bthm3}.
Then the F-pure threshold $\cc(\aaa)$ of $\aaa$
is described as
$$\cc(\aaa)= \sup_{u \in \sigma^{\vee} \setminus
\OO} \lambda_{\aaa}(u),
$$
where
$$\OO:= \{u \in \sigma^{\vee}|
\exists j, \ \langle u, v_j \rangle \ge 1\}.$$
\end{Cor}

\begin{proof}
First, we assume that $\cc(\aaa)<\sup\lambda_{\aaa}(u)$.
Then there exists $\alpha \in \RRR_{\ge 0}$ such that
$$\cc(\aaa)< \alpha <\sup\lambda_{\aaa}(u).$$
By the definition of F-pure thresholds,
$\tau(\aaa^{\alpha}) \subsetneq R$.
Then there exists \mbox{$\beta \in \RRR_{\ge 0}$} such that
$$\alpha < \beta < \sup \lambda_{\aaa}(u)$$
and $\beta=\lambda_{\aaa}(u')$ 
for $u' \in \sigma^{\vee} \setminus \OO$.
This implies that $u' \in \beta \PP(\aaa)$.
In particular, $u' \in {\rm{Int} }(\alpha \PP(\aaa))$.
In addition, $\langle u', v_j \rangle <1$ for all $j$.
By Theorem \ref{Bthm3}, it contradicts that $\tau(\aaa^{\alpha})
 \subsetneq R$.
Therefore $\cc(\aaa) \ge \sup\lambda_{\aaa}(u)$.
Second, we assume  $\cc(\aaa)>\sup\lambda_{\aaa}(u)$.
There exists $\alpha \in \RRR_{\ge 0}$ such that
$$\sup \lambda_{\aaa}(u) < \alpha < \cc(\aaa)$$
and $\tau(\aaa^{\alpha})=R$.
This implies that
there exists $\omega \in \sigma^{\vee}$ such that
$\langle \omega, v_j \rangle \le 1$ for all $j$
and 
$$\omega \in \Int(\alpha \PP(\aaa)).$$
For $1 > \varepsilon > 0$,
we have $\langle (1-\varepsilon)\omega, v_j \rangle 
= 1-\varepsilon < 1$.
Thus $(1-\varepsilon)\omega \in \sigma^{\vee} 
\setminus \OO$.
On the other hand, since 
$\omega \in \Int(\alpha \PP(\aaa))$, 
it holds that 
$$(1-\varepsilon)\omega \in \alpha \PP(\aaa),$$
for sufficiently small $\varepsilon$.
Therefore
$$\sup_{u \in \sigma^{\vee} \setminus \OO} \lambda_{\aaa}(u) 
< \lambda_{\aaa}((1-\varepsilon)\omega),$$
which is a contradiction.
Thus $\cc(\aaa) \ge \sup \lambda_{\aaa}(u)$,
which completes the proof of the corollary.
\end{proof}

Using this presentation, we compare an F-pure threshold with
an \mbox{F-threshold} with respect to the maximal monomial ideal on a toric ring.

\begin{Prop}\label{fpurelefthre}
Let $R,\ \sigma$ and $\aaa$ be as in Corollary \ref{fpuretoric},
and $\mmm$ the maximal monomial ideal of $R$.
Then
$$\cc(\aaa) \le \cc^{\mmm}(\aaa).$$
\end{Prop}

\begin{proof}
By the definitions, it is enough to show that
 $\QQ(\mmm) \subseteq \OO$.
In particular, it is enough to show $\QQ(\mmm) \cap M
\subseteq \OO$.
It follows immediately.
\end{proof}

\begin{Rem}
In general, for an ideal $\aaa$,
we have $\cc^{J'}(\aaa)\le\cc^J(\aaa)$,
where $J$ and $J'$ are ideals with $J \subseteq J'$ and
$\aaa \subseteq \sqrt{J}$.
Therefore
the F-pure threshold of $\aaa$ is less than or equal to
all F-thresholds of $\aaa$.
\end{Rem}

Now we generalize this comparison to arbitrary F-jumping
coefficients and F-thresholds.

\begin{Lem}\label{dualandpoly}
Let $R,\ \sigma$ and $\aaa$ be as in Theorem \ref{Bthm3} 
and $\omega,\ \omega' \in \sigma^{\vee}$. 
For all $j=1, \cdots, n$, we assume that
$$\langle\omega, v_j \rangle \le \langle\omega', v_j \rangle .$$
Then $\lambda_{\aaa}(\omega) \le \lambda_{\aaa}(\omega')$.
\end{Lem}

\begin{proof}
If $\lambda_{\aaa}(\omega)=0$, it is trivial.
We prove this lemma in the case $\lambda_{\aaa}(\omega) \neq 0$.
By the assumption, there exists $\omega'' \in \sigma^{\vee}$
such that $\omega'=\omega + \omega''$.
Let $\lambda:=\lambda_{\aaa}(\omega)$.
Since $\omega/\lambda \in \PP(\aaa)$,
\begin{equation*}
\frac{\omega'}{\lambda}=\frac{\omega}{\lambda}
+\frac{\omega''}{\lambda} \in \PP(\aaa)+\sigma^{\vee}.
\end{equation*}
By Proposition \ref{pro_of_PQ} (ii),
we have $\omega'/\lambda \in \PP(\aaa)$.
Hence $\lambda \le \lambda_{\aaa}(\omega')$.
\end{proof}

\begin{Prop}\label{fjumptoric}
Let $R,\ \sigma$ and $\aaa$ be as in Theorem \ref{Bthm3}.
For $u \in \sigma^{\vee} \cap M$,
we define the nonnegative number $\mu_{\aaa}(u)$ as
$$\mu_{\aaa}(u):=\sup_{\omega \in \sigma^{\vee}\setminus \OO}
\lambda_{\aaa}(u+\omega),$$
and the nonnegative number $\cc^i(\aaa)$ as
$$\cc^i(\aaa)= \inf_{X^u \in 
\tau(\aaa^{\cc^{i-1}(\aaa)})}
\mu_{\aaa}(u),$$
where $\cc^0(\aaa):=0$.
Then $\cc^i(\aaa)$ is the $i$-th F-jumping coefficient of $\aaa$.
\end{Prop}

\begin{proof}
We show that $\cc^i(\aaa)$ 
is a jumping number of the test ideal.
We assume that 
$$\tau(\aaa^{\cc^{i-1}(\aaa)})
=(X^{\mathbf{b}_1}, \cdots, X^{\mathbf{b}_t}).$$
By Lemma \ref{dualandpoly},
$$\cc^i(\aaa)= \inf_{j=1, \cdots, t} \mu_{\aaa}(\mathbf{b}_j).$$
Since $\{\mathbf{b}_j\}$ is a finite set,
there exists $j'$ such that 
$\cc^i(\aaa)=\mu_{\aaa}(\mathbf{b}_{j'})$.
By the definition of $\cc^i(\aaa)$,
for all $\omega \in \sigma^{\vee}\setminus\OO$,
$$\mathbf{b}_{j'}+\omega \notin
 \Int(c^i(\aaa)\PP(\aaa)).$$
This implies that $X^{\mathbf{b}_{j'}} \notin
 \tau(\aaa^{\cc^i(\aaa)})$ by Theorem \ref{Bthm3}.
On the other hand,
there exists $\omega' \in \sigma^{\vee}\setminus\OO$
such that
$$\mathbf{b}_{j'}+\omega' 
\in \Int((\cc^i(\aaa)-\varepsilon)\PP(\aaa)),$$
for all $\varepsilon >0$.
This also implies that $X^{\mathbf{b}_{j'}} \in
 \tau(\aaa^{\cc^i(\aaa)-\varepsilon}).$
Therefore
\mbox{$\tau(\aaa^{\cc^i(\aaa)}) \subsetneq
\tau(\aaa^{\cc^i(\aaa)-\varepsilon })$}
and hence $\cc^i(\aaa)$ is a jumping number.

We show that $\cc^i(\aaa)$ is the $i$-th F-jumping coefficient of
 $\aaa$.
In other words,
 \mbox{$\tau(\aaa^{\cc^i(\aaa)-\varepsilon})=\tau(\aaa^{\cc^{i-1}(\aaa)})$}
for all $\varepsilon >0$ with $\cc^{i-1}(\aaa) \le
 \cc^i(\aaa)-\varepsilon$.
The inclusion  $\tau(\aaa^{\cc^i(\aaa)-\varepsilon})
\subseteq \tau(\aaa^{\cc^{i-1}(\aaa)})$ follows immediately
from Theorem \ref{Bthm3}.
The opposite inclusion 
follows from the definition of $\cc^i(\aaa)$.
In fact, if $X^u \in \tau(\aaa^{\cc^{i-1}(\aaa)})$,
then \mbox{$\cc^i(\aaa)-\varepsilon < \cc^i(\aaa) \le \mu_{\aaa}(u)$,}
by definition of $\cc^i(\aaa)$.
Hence there exists $\omega \in \sigma^{\vee}\setminus \OO$
such that
$$u+\omega \in \Int((\cc^i(\aaa)-\varepsilon)\PP(\aaa)).$$
This implies that $X^u \in \tau(\aaa^{\cc^i(\aaa)-\varepsilon})$
by Theorem \ref{Bthm3}.
We complete the proof of the proposition.
\end{proof}

\begin{Prop}
We have the following inequality:
$$\cc^i(\aaa) \le 
\cc^{\tau(\aaa^{\cc^i(\aaa)})}(\aaa).$$
\end{Prop}

\begin{proof}
Since $\tau(\aaa^{\cc^i(\aaa)}) \subsetneq 
\tau(\aaa^{\cc^{i-1}(\aaa)})$,
there exists $u \in \sigma^{\vee} \cap M$ such that
\mbox{$X^u \in \tau(\aaa^{\cc^{i-1}(\aaa)})$}
and $X^u \notin \tau(\aaa^{\cc^i(\aaa)})$.
By Proposition \ref{fjumptoric},
\begin{equation}\label{**}
\cc^i(\aaa) \le \mu_{\aaa}(u).
\end{equation}
We claim  that for all $\omega \in \sigma^{\vee}\setminus \OO$,
\begin{equation*}
\omega+u \in \sigma^{\vee} \setminus \QQ(\tau(\aaa^
{\cc^i(\aaa)})).
\end{equation*}
By Theorem \ref{Thmfthre}, this claim implies that
\begin{equation}\label{*4}
\mu_{\aaa}(u) \le \cc^{\tau(\aaa^{\cc^
i(\aaa)})}(\aaa).
\end{equation}
The proof of the proposition is completed from inequalities
 (\ref{**}) and (\ref{*4}).
Now we prove that claim.
We assume that there exists $\omega \in \sigma^{\vee}\setminus \OO$
 such that $u+\omega \in \QQ(\tau(\aaa^
{\cc^i(\aaa)}))$.
There exist $u' \in M$ and $\omega' \in \sigma^{\vee}$
such that $X^{u'} \in \tau(\aaa^{\cc^i (\aaa)})$ 
and $u+\omega=u'+\omega'$.
Thus $u-u'=\omega'-\omega \in M$.
On the other hand, since $u=u'+\omega'-\omega \in M$
and $X^u \notin \tau(\aaa^{\cc^i(\aaa)})$,
we have $\omega'-\omega \notin \sigma^{\vee}$.
That is, there exists $j$ such that $\langle (\omega'-\omega),
 v_j \rangle < 0$.
Therefore
$$0 \le \langle \omega', v_j \rangle <
\langle \omega, v_j \rangle < 1.$$
It contradicts that $\omega'-\omega \in M$.
Hence we have the claim, 
and then we complete the proof of the proposition.
\end{proof}

\begin{Rem}
Since a toric ring is strongly F-regular,
$\aaa \subseteq 
\tau(\aaa^{\cc^i(\aaa)})$.
Hence $\cc^{\tau(\aaa^{\cc^
i(\aaa)})}(\aaa)$ exists and is a finite number.
\end{Rem}

\section{Applications}
Let us give some applications of results in previous sections.
As we see in Corollary \ref{threisjump},
for an arbitrary ideal $\aaa$,
F-thresholds of $\aaa$ are equal to 
\mbox{F-jumping} coefficients of $\aaa$ on regular rings.
By the formula of F-thresholds,
we see that
if $R$ is a toric ring
 which has at most Gorenstein
singularities,
then there exists a monomial ideal $\aaa \subseteq R$
such that \mbox{$\cc(\aaa)=\cc^{\mmm}(\aaa)$.}

\begin{Prop}\label{comparigoren}
Let $R$ be a Gorenstein toric ring defined by
a cone \mbox{$\sigma \subseteq N_{\RRR}=\RRR^d$} and 
$\mmm$ the maximal monomial ideal.
There exist a monomial ideal $\aaa \subseteq R$
such that $\cc(\aaa)=\cc^{\mmm}(\aaa)$.
\end{Prop}

\begin{proof}
We assume that $\sigma =
 \RRR_{\ge 0}v_1 + \cdots + \RRR_{\ge 0}v_n$,
where $v_j \in N$ are primitive numbers. 
For a Gorenstein toric ring $R$,
there exists an element $\omega \in \sigma^{\vee}\cap M$ such that
$\langle \omega, v_j \rangle =1$ for all $j=1, \cdots, n$.
By Lemma \ref{dualandpoly}, for a monomial ideal $\aaa \subseteq R$,
we can describe
$$\cc(\aaa) = \lambda_{\aaa}(\omega).$$ 
Let $\aaa=(X^{\omega})$.
We have $\PP(\aaa)=\omega + \sigma^{\vee}$,
and clearly $\cc(\aaa)=\lambda_{\aaa}(\omega)=1$.
Since $\omega \in M \setminus \mathbf{0}$, we have $\omega \in
 \QQ(\mmm)$.
Hence $\PP(\aaa) \subseteq \QQ(\mmm)$.
By Theorem \ref{Thmfthre},
that implies $\cc^{\mmm}(\aaa) \le 1=\cc(\aaa)$.
On the other hand,
$\cc(\aaa)\le\cc^{\mmm}(\aaa)$ follows by Proposition \ref{fpurelefthre}.
We complete the proof of the proposition.
\end{proof}

For $2$-dimensional toric rings,
we see that the opposite assertion of Proposition \ref{comparigoren}
is true.
However, it is false in general toric rings whose dimension
 are greater than $3$.

\begin{Prop}
Let $R$ be a $2$-dimensional toric ring, 
and $\mmm$ the maximal monomial ideal of $R$.
If there exists a monomial ideal $\aaa \subseteq R$
such that $\cc(\aaa)=\cc^{\mmm}(\aaa)$,
then $R$ has at most Gorenstein singularities.
\end{Prop}

\begin{proof}
Suppose that $R$ is defined by a cone $\sigma$.
By taking a suitable change of coordinates,
it suffices to consider cones $\sigma:=\RRR_{\ge 0}(1,0)+\RRR_{\ge 0}(a,b)
 \subseteq \RRR^2$, where $b >0$ and the greatest common divisor of $a$
 and $b$ is $1$.
The following three cases are trivial.
If $a=0$,
then $R$ is the polynomial ring.
If $a=1, b=1$, then $R=k[X_1,\ X_1^{-1}X_2]$,
which is a regular ring.
If $a=1$, $b>1$,
then $R=k[X_1, X_2, X_1^{b}X_2^{-1}]\cong k[x,y,z]/(xz-y^{b})$.
Note that ${\rm{Spec}} R$ has an ${\sc{A}}_{b-1}$ singularity.
Hence $R$ is a Gorenstein ring.
Assume that $a>1$.
We have $\sigma^{\vee}=\RRR_{\ge 0}(0,1)+\RRR_{\ge 0}(b,-a)$,
and the point $\omega=(1,(1-a)/b)$ which satisfies
$$\langle \omega, (1,0)\rangle=\langle \omega, (a,b)\rangle=1.$$
If $\omega \notin \QQ(\mmm)$,
then for all monomial ideals $\aaa$, we have
$\cc(\aaa)<\cc^{\mmm}(\aaa)$.
In fact, by taking $\varepsilon >0$ 
with $(1+\varepsilon)\omega \notin \QQ(\mmm)$,
we have a strict inequality;
$$\cc(\aaa)<\lambda_{\aaa}((1+\varepsilon))\le \cc^{\mmm}(\aaa).$$
By the assumption of the proposition,
$\omega \in \QQ(\mmm)$.
Thus it is enough to prove that
$\omega \in M$ under the assumption $\omega \in \QQ(\mmm)$.
By the definition of $\QQ(\mmm)$,
if $\omega \in \QQ(\mmm)$,
then there exists a lattice point 
\mbox{$u \in \sigma^{\vee} \cap M \setminus \{\mathbf{0}\}$}
such that $\omega - u \in \sigma^{\vee}$.
Since $u \in \sigma^{\vee}$,
the lattice point $u$ is written as
$u=\lambda_1(0,1)+\lambda_2(b,-a) \in M$,
where $\lambda_1$ and $\lambda_2$ are positive.
Since $\omega - u \in \sigma^{\vee}$,
we have $(1/b)-\lambda_1 \ge 0$ and $(1/b)-\lambda_2 \ge 0$.
Since $\mathbf{0} \neq u \in M$ and $b \in \ZZZ_{>0}$,
we have $\lambda_2=1/b$.
Hence $u=(1, \lambda_1 -(a/b))$.
Since $u \in M$,
there exists an integer $l$
such that $l=\lambda_1-(a/b)$ and 
$$-\frac{a}{b} \le l \le \frac{1-a}{b}.$$
Since $a,\ b \in \ZZZ$ and the greatest common divisor of
$a$ and $b$ is $1$,
we have $bl=1-a$.
Thus $b|(1-a)$.
This implies that $\omega \in M$.
The remaining case when $a<0$ follows by the same argument.
We complete the proof of the proposition.
\end{proof}

\begin{Ex}\label{Exr-goren}
Suppose $N=\RRR^3$.
We define generators $\{v_i\}$ of a cone $\sigma \subseteq N_{\RRR}$
as
$$v_1:=(1,0,0),\
v_2:=(1,1,0),\
v_3:=(0,1,r).$$
Let $\omega:=(1,0,1/r) \in \sigma^{\vee}$.
Since $\langle \omega,\ v_i \rangle=1$ for all $i$,
the toric ring $R$ defined by $\sigma$ has an $r$-Gorenstein singularity.
We choose a set of generators of $\sigma^{\vee}$ as
$$u_1:=(r,-r,1),\
u_2:=(0,r,-1),\
u_3:=(0,0,1).$$
Then
$$\omega=\frac{1}{r}u_1+\frac{1}{r}u_2+\frac{1}{r}u_3.$$
Since $\omega-(1/r)u_3 \in \sigma^{\vee}\cap M$,
we have $\omega\in \QQ(\mmm)$,
where $\mmm \subseteq R$ is the maximal monomial ideal.
Let $\aaa$ be a monomial ideal generated by $X^{r\omega}$.
Then $(1/r)\PP(\aaa)=\omega+\sigma^{\vee} \subseteq \QQ(\mmm)$.
The same argument in the proof of Proposition \ref{comparigoren}
implies $\cc(\aaa)=\cc^{\mmm}(\aaa)=1/r$.
\end{Ex}

\begin{Ex}
Suppose $N=\RRR^d$, where $d> 3$.
We consider the cone $\sigma$ generated by
\begin{eqnarray*}
v_1:=(1,0,0,0,&\cdots &,0)\\
v_2:=(1,1,0,0,&\cdots &,0)\\
v_3:=(0,1,r,0,&\cdots &,0)\\
v_i:=(0,0,0,0,&\cdots,0,\stackrel{i}{\breve{1}},0,\cdots &,0),\ 3<i\le d.
\end{eqnarray*}
By the same argument in Example \ref{Exr-goren},
we have a monomial ideal $\aaa$ of a $d$-dimensional 
$r$-Gorenstein ring $R=k[\sigma^{\vee}\cap M]$
such that $\cc(\aaa)=\cc^{\mmm}(\aaa)$.
\end{Ex}

Using F-thresholds and F-pure thresholds,
we give a criterion of regularities for a toric ring defined 
by a simplicial cone.

\begin{Thm}\label{charareg}
Let $R$ be a toric ring defined by 
a simplicial cone $\sigma$,
and $\mmm$ the maximal monomial ideal.
If there exists a monomial ideal $\aaa$ such that
$\sqrt{\aaa} = \mmm$ and 
$$\cc(\aaa) = \cc^{\mmm}(\aaa),$$
then $R$ is a regular ring.

\end{Thm}

\begin{proof}
Let $\sigma \subseteq N_{\RRR}:= \RRR^d$.
Since $\sigma$ is simplicial,
we may assume that
$$\sigma=\RRR_{\ge 0}v_1 + \cdots \RRR_{\ge 0}v_d,$$
where $v_j \in N$ and $\{v_1,\cdots,v_d\}$ 
are $\RRR$-linearly independent.
Hence there exist $u_i \in M$ and $l_i \in \ZZZ_{>0}$
such that 
$$\sigma^{\vee}=\RRR_{\ge 0}u_1 + \cdots + \RRR_{\ge
 0}u_d,$$
and $\langle u_i, v_j \rangle = l_i \delta_{ij}$.
Moreover, for all $i, j =1, \cdots , d$,
we assume $v_j$ and $u_i$ are primitive.
Since $\sigma$ is simplicial, $R$ is $\QQQ$-Gorenstein.
Hence there exists $\omega \in M \otimes \QQQ$ such that
$$\cc(\aaa)=\cc^{\mmm}(\aaa)
=\lambda_{\aaa}(\omega).$$
By Theorem \ref{Thmfthre},
\begin{equation}\label{cri}
\lambda_{\aaa}(\omega)\PP(\aaa) 
\subseteq \QQ(\mmm).
\end{equation}
To prove the theorem,
it is enough to show that $l_i=1$ for every $i=1,\cdots,d$.
We derive a contradiction assuming $l_i > 1$ for some $i$.
Since $\sqrt{\aaa} = \mmm$,
for a sufficiently large nonnegative integer $l$,
we have $X^{lu_i} \in \aaa$.
In particular,
$\lambda_{\aaa}(\omega)lu_i \in \lambda_{\aaa}(\omega)\PP(\aaa)$.
If we choose sufficiently large $l$,
then we have
$$0 < \frac{l_i - 1}{\lambda_{\aaa}(\omega) l l_i - 1} < 1.$$
Let $\alpha \in \RRR_{>0}$ such that
$0 < \alpha < (l_i - 1)/(\lambda_{\aaa}(\omega) l l_i - 1)$.
By the definition of $\PP(\aaa)$ and (\ref{cri}),
$$\alpha \lambda_{\aaa}(\omega) l u_i + (1 - \alpha) \omega 
\in \QQ(\mmm).$$
On the other hand, for all $j$,
\begin{eqnarray*}
\langle \alpha \lambda_{\aaa}(\omega) l u_i + (1 - \alpha) \omega, v_j \rangle =
\left\{\begin{array}{ll}
1 - \alpha < 1
& (j \neq i), \\
\alpha \lambda_{\aaa}(\omega) l l_i + 1 - \alpha  < l_i
& (j=i).
\end{array} \right.
\end{eqnarray*}
By the definition of $\QQ(\mmm)$,
there exist $l'_i \in \ZZZ_{> 0}$,
$u \in M \cap \QQ(\mmm)$ and $u' \in \sigma^{\vee}$
such that 
\begin{eqnarray*}
\langle u, v_j \rangle =
\left\{\begin{array}{ll}
0 & (j \neq i) \\
l_i' < l_i
& (j=i),
\end{array} \right.
\end{eqnarray*}
and 
$$\alpha \lambda_{\aaa}(\omega) l u_i + (1 - \alpha)\omega = u + u'.$$
However,
the existence of $u$ contradicts the primitiveness of $u_i$.
Thus $l_i = 1$.
Eventually, for every $i= 1,\cdots, d$, we have $l_i=1$.
Therefore we complete the proof of the theorem.
\end{proof}

On the other hand, there exist a toric $R$
defined by a non-simplicial cone
and a maximal ideal $\mmm$
such that $\cc(\mmm)=
\cc^{\mmm}(\mmm)$.

\begin{Ex}[{\cite[Remark 2.5]{HMTW}}]
If $R=k[X_1X_3,X_2X_3,X_3,X_1X_2X_3]$
and $\mmm=(X_1X_3,X_2X_3,X_3,X_1X_2X_3)$,
then $R$ is a toric ring whose defining cone is 
$$\sigma=
\RRR_{\ge 0}(1,0,0)+
\RRR_{\ge 0}(0,1,0)+
\RRR_{\ge 0}(-1,0,1)+
\RRR_{\ge 0}(0,-1,1).
$$
There exists $\omega=(1,1,2) \in \sigma^{\vee}$
which entails
$$
\langle \omega, (1,0,0)\rangle=
\langle \omega, (0,1,0)\rangle=
\langle \omega, (-1,0,1)\rangle=
\langle \omega, (0,-1,1)\rangle=1.
$$
By Corollary \ref{fpuretoric} and Lemma \ref{dualandpoly},
for every monomial ideal $\aaa$,
we have \mbox{$\cc(\aaa)=\lambda_{\aaa}(\omega)$.}
Hence $\cc(\mmm)=2$.
On the other hand, we can compute 
\mbox{$\cc^{\mmm}(\mmm)=2$.}
\end{Ex}

Finally, we discuss about the rationality of F-thresholds.
This was given as an open problem in \cite{MTW}.
For some regular rings,
Blickle, Musta\c{t}\v{a} and Smith give the affirmative answer. 
In \cite{BMS2}, they prove the rationality of
F-thresholds of all proper ideals $\aaa$ 
with respect to ideals $J$ which entail \mbox{$\aaa \subseteq \sqrt{J}$}
on an F-finite regular ring essentially of finite type over $k$
(\cite[Theorem 3.1]{BMS2}).
In addition, they also prove in cases that
$\aaa=(f)$ is principal on an \mbox{F-finite} regular ring
(\cite[Theorem 1.2]{BMS1}).
On the other hand,
Katzman, Lyubeznik and Zhang
prove in cases that $\aaa = (f)$ is principal
on an excellent regular local ring,
that is not necessarily F-finite
(\cite{KLZ}). 
We will prove rationality of an F-threshold of a monomial ideal $\aaa$ 
with respect to an $\mmm$-primary monomial ideal $J$ on a toric ring.
This argument is described in terms of real affine geometries.
We define the affine half space $\HH(v; \lambda)$ as
$$\HH(v;\lambda):=\{u \in M_{\RRR}|
\langle u, v \rangle \ge \lambda \},$$
where  $v \in N_{\RRR}$ and $\lambda \in \RRR$.
We also define the hyperplane $\partial \HH(v; \lambda)$ as
$$\partial \HH(v; \lambda):= \{u \in M_{\RRR}|
\langle u, v \rangle = \lambda\}.$$
Assume that $\aaa$ is a monomial ideal of a toric ring.
Since $\PP(\aaa)$ is a convex polyhedral set,
it is written as an intersection of finite affine half spaces.
we observe the form of $\PP(\aaa)$.
 
\begin{Lem}\label{polygon}
Let $R$ be a toric ring defined by a cone 
$\sigma \subseteq N_{\RRR}=\RRR^d$,
and $\aaa$ a monomial ideal of $R$.
Then there exist $v'_l \in N_{\QQQ}:=N\otimes \QQQ$ and
$\lambda'_l \in \QQQ$ for $l=1, \cdots, t$
such that
$\PP(\aaa)= \cap_{l=1}^{t}\HH
(v'_l;\lambda'_l)$.
\end{Lem}

\begin{proof}

Since $\sigma$ is a rational polyhedral cone,
so is $\sigma^{\vee}$.
Hence there exists $u_i \in M$ such that
$$\sigma^{\vee} = \RRR_{\ge 0}u_1 + 
\cdots + \RRR_{\ge 0}u_m.$$
We assume that $\aaa=(X^{\mathbf{a}_1},\cdots, X^{\mathbf{a}_s})$.
We consider the rational polyhedral cone $\tau$ of $M_{\RRR}
 \times \RRR$ as
$$\tau:=\RRR_{\ge 0}(\mathbf{a}_1,1)+
\cdots + \RRR_{\ge 0}(\mathbf{a}_s,1)+
\RRR_{\ge 0}(u_1,0)+ 
\cdots + \RRR_{\ge 0}(u_m,0).$$
For such $\tau$ and $\PP(\aaa)$,
\begin{equation}\label{polar1}
\tau \cap (M_{\RRR} \times \{1\}) = 
\PP(\aaa) \times \{1\}.
\end{equation} 
In fact, let $(u,1)$ be an element of the left-hand side.
Then
$$(u,1)=\sum_{i=1}^s a_i (\mathbf{a}_i,1) + \sum_{j=1}^m b_j(u_j,0),$$
where $a_i, b_j \ge 0$.
By the definition, $\sum a_i = 1$.
By Proposition \ref{pro_of_PQ} (iii), $u \in \PP(\aaa)$.
The similar argument implies the opposite inclusion.
Since $\tau$ is the rational polyhedral convex cone,
for $l = 1,\cdots,t$,
there exists 
\newline
$(v'_l, \mu_l) \in N_{\QQQ} \times \QQQ$
such that
\begin{equation}\label{polar2}
\tau = \bigcap_{l=1}^t \HH((v'_l,\mu_l);0),
\end{equation}
where $\HH((v'_l,\mu_l);0)$ is the affine half space of 
$M_{\RRR}\times \RRR$.
The duality pair of $M_{\RRR} \times \RRR$ and 
$N_{\RRR} \times \RRR$ is defined as
$$\langle (u,\lambda), (v,\mu) \rangle := \langle u,v \rangle 
+ \lambda \mu,$$
for every $u \in M_{\RRR},\ v \in N_{\RRR}$
and $\lambda ,\ \mu \in \RRR$.
Under this duality,
$$\HH((v,\mu);0) \cap (M_{\RRR} \times \{1\})
= \HH(v;-\mu) \times \{1\}.$$
Therefore if we set $\lambda'_l:=-u_l$ for each $l=1,\cdots, t$,
the assertion of the theorem follows by (\ref{polar1}) and (\ref{polar2}). 
\end{proof}

\begin{Thm}\label{rational}
Let $R,\ \sigma$ and $\aaa$ be as in Lemma \ref{polygon}.
Furthermore, we assume that $\sigma$ is a $d$-dimensional simplicial
 cone.
Let $J$ be an $\mmm$-primary monomial ideal,
where $\mmm$ is the maximal monomial ideal of $R$.
Then the F-threshold $\cc^J(\aaa)$ of $\aaa$ with respect to $J$
is a rational number.
\end{Thm}

\begin{proof}
We denote by $\partial \QQ(J)$ 
the boundary of $\QQ(J)$ in $\sigma^{\vee}$.
By Lemma \ref{polygon},
if there exists a finite set $B \subseteq M_{\QQQ} \cap \partial \QQ(J)$
such that 
$$\cc^J(\aaa)= \max_{\omega \in B} \lambda_{\aaa}(\omega),$$ 
then we have $\cc^J(\aaa) \in \QQQ$.
\newline
First, we prove that  
$$\cc^J(\aaa)=\sup_{\omega \in \partial\QQ(J)} \lambda_{\aaa}(\omega).$$
By Theorem \ref{Thmfthre},
if there exists an element $\omega \in \sigma^{\vee}$
such that $\cc^J(\aaa)= \lambda_{\aaa}(\omega)$,
then $\omega \in \partial \QQ(J)$.
In fact, if such $\omega$ is in $\sigma^{\vee} \setminus \QQ(J)$,
there exists $\varepsilon > 0$ such that
 $(1+\varepsilon)\omega \in \sigma^{\vee} \setminus \QQ(J)$.
This implies that $\cc^J(\aaa) \ge
 (1+\varepsilon)\lambda_{\aaa}(\omega)$.
It is a contradiction, thus we are done.
\newline
Second, we prove the existence of $B \subseteq M_{\QQQ} \cap
\partial \QQ(J)$.
We assume that $\sigma=\RRR_{\ge 0}v_1+\cdots+\RRR_{\ge
 0}v_{d}$, where $v_j$ are primitive lattice points.
Since $\sigma$ is simplicial, for every $j$,
there exists $u_j \in M_{\QQQ}$ such that 
$$\langle u_j, v_l \rangle = \delta_{j l},\ l \in \{1,\cdots,d\}.$$
Since $J$ is $\mmm$-primary,
there exists a nonnegative integer $r_j$
such that \mbox{$r_j u_j \in \QQ(J)$.}
That implies $\partial \QQ(J)$ is bounded.
We define the order $\le_{\sigma}$ over $\partial \QQ(J)$ 
as $u \le_{\sigma} u' $ if
$$\langle u, v_j \rangle \le 
\langle u', v_j \rangle,\ \forall j= 1,\cdots, d.$$
Then $\partial \QQ(J)$ has maximal elements with respect to
this order.
Let \mbox{$B \subseteq \partial \QQ(J)$} be the set of 
maximal elements with respect to the order $\le_{\sigma}$.
By Lemma \ref{dualandpoly},
we conclude
$$\cc^J(\aaa)=
\sup_{\omega \in \partial\QQ(J)} \lambda_{\aaa}(\omega)=
\sup_{\omega \in B} \lambda_{\aaa}(\omega).$$
To show that $B$ is a finite set of $M_{\QQQ}$,
we prove the following claim.
\begin{Cl*}
Let $J=(X^{\mathbf{b}_1},\cdots,X^{\mathbf{b}_t})$.
We assume that $u\in B$, that is,
\begin{description}
\item[(i).] $u \in \partial \QQ(J)$,
\item[(ii).] $u$ is a maximal element with respect to the order $\le_{\sigma}$
in $\partial \QQ(J)$.
\end{description}
Then for every $j=1,\cdots,d$,
there exists $i_j$ 
such that 
\begin{equation}\label{pro0}
u \in \bigcap_{j=1}^n \left(\mathbf{b}_{i_j} +(\partial \HH
(v_j;0)\cap \sigma^{\vee})\right).
\end{equation}
In particular, $B$ is a finite set and $u \in M_{\QQQ}.$
\end{Cl*}
\begin{proof}[Proof of Claim]
We suppose that $u$ does not satisfy (\ref{pro0}).
Then there exists $j' \in \{1,\cdots, d\}$ such that 
\begin{equation}\label{pro3}
u \notin \mathbf{b}_i + 
(\partial\HH(v_{j'};0) \cap \sigma^{\vee}),
\end{equation}
for all $i=1,\cdots,t$.
We choose $u' \in \sigma^{\vee}$ as
$$\langle u', v_j \rangle = \langle u, v_j \rangle,\ (j \neq j'),$$
$$\langle u', v_{j'} \rangle = \lfloor \langle u, v_{j'} \rangle \rfloor
 +1.$$
Since $\sigma$ is simplicial, $u'$ uniquely exists.
We will show that the existence of $u'$ contradicts
the assumption (ii).
By the construction of $u'$,
we have $u' \in \QQ(J)$.
To see $u' \notin \Int \QQ(J)$,
we rephrase the assumption (i).
Since $u \notin \Int\QQ(J)$,
we have $u \notin \mathbf{b}_i + \Int(\sigma^{\vee})$
 for all $i=1,\cdots,t$.
Furthermore, this is equivalent to
the existence of $l_i$ such that
\begin{equation}\label{pro2}
\langle u, v_{l_i} \rangle \le \langle \mathbf{b}_i, v_{l_i}
\rangle,
\end{equation}
 for each $i=1, \cdots, t$.
If $l_i \neq j'$,
we have directly
$$\langle u', v_{l_i} \rangle 
= \langle u, v_{l_i} \rangle \le 
\langle \mathbf{b}_{i}, v_{l_i} \rangle,$$
by the construction of $u'$ and the relation (\ref{pro2}).
On the other hand,
if $l_i=j'$,
then the relation (\ref{pro2}) and (\ref{pro3}) implies 
$$\lfloor \langle u, v_{j'} \rangle \rfloor
\le \langle \mathbf{b}_i, v_{j'} \rangle -1,$$
because $\mathbf{b}_i \in M$.
Hence
$\langle u', v_{l_i} \rangle \le 
\langle \mathbf{b}_{i}, v_{l_i} \rangle.$
Eventually, in both cases, \mbox{$u' \notin \Int \QQ(J)$.}
Therefore $u' \in \partial \QQ(J)$.
By the construction of $u'$,
the element $u$ is not a maximal element in $\partial \QQ(J)$.
It contradicts the assumption (ii).
We complete the proof of Claim.

\end{proof}
We complete the proof of the theorem. 
\end{proof}

Now we consider the rationality of F-jumping coefficients 
on $\QQQ$-Gorenstein toric rings.
The rationality of F-jumping coefficients is 
the consequence of the fact that test ideals are 
equal to multiplier ideals
(\cite[Theorem 4.8]{HY} and \cite[Theorem 1]{B}).
However, we also give its proof by the combinatorial method.

\begin{Prop}\label{rationaljump}
Let $R,\ \sigma,\ \aaa$ be as in Lemma \ref{polygon}.
Moreover, we assume $R$ is an $r$-Gorenstein toric ring.
Then for all $i$, the $i$-th F-jumping coefficient $\cc^i(\aaa)$
of $\aaa$ is a rational number.
\end{Prop}

\begin{proof}
In the proof of Proposition \ref{fjumptoric},
we have seen that there exists $\mathbf{b} \in M$
such that $\cc^i(\aaa)=\mu_{\aaa}(\mathbf{b})$,
where $X^{\mathbf{b}}$ is one of generators of 
$\tau(\aaa^{\cc^{i-1}(\aaa)})$.
By the similar argument to that of the proof of Proposition \ref{comparigoren},
there exists $\omega \in \sigma^{\vee}$
such that $\cc^i(\aaa)=\lambda_{\aaa}(\mathbf{b}+\omega / r)$.
Since $\omega$ corresponds to the generator of 
$\omega_R^{(r)}$, where $\omega_R$ is the canonical module of $R$,
we see $\omega \in M$.
Hence $\mathbf{b}+\omega /r \in M_{\QQQ}$.
Therefore $\cc^i(\aaa)$ is a rational number.
\end{proof}

\subsection*{Acknowledgement}
The author would like to express his thanks to Professor Kei-ichi 
Watanabe who informs him the formula of F-thresholds 
on regular toric rings.
The author also thanks to Professor Daisuke Matsushita
for his constant advice and encouragement.


\end{document}